\documentclass{article}
\usepackage[utf8]{inputenc}

\usepackage{amsmath, amssymb, amsthm, amsfonts}
\usepackage{csquotes} 
\usepackage[american]{babel}
\usepackage{hyperref} 
\usepackage{color}
\usepackage{enumitem} 

\newtheorem{theorem}{Theorem}
\newtheorem{lemma}[theorem]{Lemma}

\numberwithin{theorem}{section}
\numberwithin{equation}{section}

\definecolor{darkgreen}{rgb}{0.0, 0.55, 0.55}

\newcommand{\Z}{\mathbb{Z}}

\newcommand{\R}{\mathbb{R}}

\newcommand{\T}{\mathbb{T}}

\newcommand{\real}{\operatorname{Re}}

\begin{document}

\title{Global wellposedness of NLS in $H^1(\R) + H^s(\T)$}
\author{Friedrich Klaus\footnote{Karlsruhe Institute of Technology, friedrich.klaus@kit.edu}~ and Peer Kunstmann\footnote{Karlsruhe Institute of Technology, peer.kunstmann@kit.edu}}
\date{September 2021}

\maketitle

\begin{abstract}
    We show global wellposedness for the defocusing cubic nonlinear Schrö\-dinger equation (NLS) in $H^1(\R) + H^{3/2+}(\T)$, and for the defocusing NLS with polynomial nonlinearities in $H^1(\R) + H^{5/2+}(\T)$. This complements local results for the cubic NLS \cite{teeth} and global results for the quadratic NLS \cite{qNLS} in this hybrid setting.
\end{abstract}

\section{Introduction}

We consider the wellposedness question for the nonlinear Schrödinger equation 
\begin{equation}
    \begin{split}\label{eq:NLS}
        iu_t + u_{xx} &= |u|^{p-1}u,\\
        u(0) &= u_0 \in H^{s_1}(\R) + H^{s_2}(\T)
    \end{split}
\end{equation}
with non-decaying initial data. This problem has been an area of active research for many years now. One of its motivations is the propagation of signals in glass-fiber cables, where the cubic NLS is used as an approximate model equation \cite{KIVSHAR}. In this model, the roles of space and time are reversed and the initial value $u_0$ describes the signal seen at a fixed point of the cable. Hence, periodic initial data can be understood as encoding, e.g., an infinite string of ones. Such a signal carries no information, and we want to consider signals where some of the ones have been overwritten by a zero. Following \cite{teeth} this is done by adding a nonperiodic part $v_0 \in H^{s_1}(\R)$ to the initial data $w_0 \in H^{s_2}(\T)$. Global existence then translates to having no bound on the length of the cable. 

From a mathematical point of view, there is a huge number of directions by which NLS with non-decaying initial data has been approached, and we will only name some of them. The most classical one of them is purely periodic initial data, both in the general case \cite{bourgain} and even earlier in the integrable case $p=3$ under assumptions on the spectral properties of the corresponding Lax operator \cite{ablowitz}. A natural generalisation is to consider quasi-periodic, almost periodic and limit periodic initial data \cite{faddeev, nlscantor, ohlimitperiodic}. NLS with prescribed boundary value $\lim_{x\to\pm \infty} |u(x)| = 1$ is known as the Gross-Pitaevskii equation and describes Bose gases at zero temperature \cite{KIVSHAR}. For these mentioned types of initial data, global results exist. Additionally, there are local results in the case of initial data lying in the modulation spaces $M^s_{\infty,2}(\R)$ \cite{benyi,modulationintersection} and the case of analytic initial data \cite{dodsonbounded}.

Our approach is to consider initial data which are the sum of a periodic and a decaying signal, i.e. $u_0 \in H^{s_1}(\R) + H^{s_2}(\T)$. We note that this type of data is more general than the periodic case, and also includes the emergence of dark solitons as in the Gross-Pitaevskii case, but is less general than for example $u_0 \in M^s_{\infty,2}(\R)$. Local wellposedness results for this problem have been covered in \cite{teeth, qNLS}, and our main interest is to extend them to global solutions. These local solutions are constructed as follows: By writing $u = v + w \in H^{s_1}(\R) + H^{s_2}(\T)$ and using the fact that $u$ satisfies \eqref{eq:NLS}, $w$ is seen to also satisfy \eqref{eq:NLS} but on the torus,
\begin{equation}
    \begin{split}\label{eq:NLSw}
        iw_t + w_{xx} &= |w|^{p-1}w,\\
        w(0) &= w_0 \in H^{s_2}(\T),
    \end{split}
\end{equation}
at least if $u,v,w$ have suitable regularity in space and time. Hence $v$ has to be the solution of the perturbed problem
\begin{equation}
    \begin{split}\label{eq:NLSv}
        iv_t + v_{xx} &= |v+w|^{p-1}(v+w) - |w|^{p-1} w,\\
        v(0) &= v_0 \in H^{s_1}(\R),
    \end{split}
\end{equation}
where $w \in C([0,\infty),H^{s_2}(\T))$ is the solution of \eqref{eq:NLSw}. Equation \eqref{eq:NLSv} is a perturbation of NLS on the real line, and classical methods like Strichartz estimates can be used to establish local wellposedness \cite{qNLS}.

In order to extend local to global solutions, we only need to consider \eqref{eq:NLSv}, because \eqref{eq:NLSw} is known to exhibit global solutions. The main problem here is that the conservation laws that exist both in the periodic and non-periodic case give rise to a conservation law for \eqref{eq:NLSv} with an $L^1$ part in it, for example we have formal conservation of
\[
    \int_\R |u|^2 - |w|^2 \, dx = \int_\R |v|^2 + 2\real(v\bar w)\, dx,
\]
but not of $\int |v|^2\, dx$. As a consequence, these exact conservation laws are not applicable in the $L^2$-based setting. Instead, we want to make use of quantities for which we can control the growth rate. In this setting, the power of the nonlinearity plays a crucial role. Indeed, for the quadratic nonlinearity $p=2$ \cite{qNLS} showed global wellposedness in low regularity by a Gronwall argument for $\int |v|^2\, dx$. Such a straight forward calculation does not work anymore if $p > 2$.

To overcome this problem, we will work with higher regularity and assume $v_0 \in H^1(\R)$. Note that the time-dependent Hamiltonian
\begin{equation}\label{eq:differenceHamiltonian}
    H(q,r) = \int_{\R} q_x r_x + (q+w)^2(r+\bar w)^2- |w|^4 - 2(q\bar w + r w) |w|^2 \, dx,
\end{equation}
gives rise to \eqref{eq:NLSv} as the Hamiltonian equation
\begin{equation*}
    iq_t = \frac{\delta}{\delta r} H(q,r) \Big|_{(q,r) = (v,\bar v)}.
\end{equation*}
Our main idea is to make use of the formula
\begin{equation}\label{eq:flow}
    \frac{d}{dt} F(t,v(t)) = \{F,H\}(t,v(t)) + (\partial_t F)(t,v(t)),
\end{equation}
which holds for Hamiltonian equations with Poisson bracket $\{\cdot,\cdot\}$, and choosing $F$ to be the Hamiltonian of the equation \eqref{eq:NLSv} itself. From \eqref{eq:flow} it follows that the time derivative of $H$ with respect to the flow induced by itself is non-zero, but only contains time derivatives that fall on $w$, because this is the only explicitly time-dependent part of \eqref{eq:differenceHamiltonian}. We want to mention that these calculations were inspired by calculations performed in \cite{dodson} and the described Hamiltonian formalism makes also transparent why they work there.

While the Hamiltonian structure lies implicit in all arguments used below, we choose not to use the Hamiltonian language in what follows. The reason for this is that the main difficulty in the calculation is not to calculate the formal time derivative, but rather to make sure that taking time derivatives is allowed. Still it may help the reader in understanding why the calculations work as they stand, as it helped the authors in doing so.

The paper is organized as follows: in Section \ref{sec1} we give a proof that global solutions on the torus exist, in Section \ref{sec2} we prove that local solutions on the line exist, and in Section \ref{sec3} we show that the local solutions on the line are global. For the presentation we give simple and self-contained proofs.

\subsubsection*{Acknowledgements}
The authors want to thank Leonid Chaichenets, Dirk Hundertmark and Robert Schippa for useful discussions. Funded by the Deutsche Forschungsgemeinschaft (DFG, German Research Foundation) – Project-ID 258734477 – SFB 1173.

\section{Global periodic solutions}\label{sec1}

The wellposedness theory of NLS with periodic data is a classical problem. When $2 \leq p < 5$, Lebowitz, Rose and Speer \cite{lebowitzrosespeer} (refering to \cite{ginibrevelo} for a proof) and later Bourgain \cite{bourgain} in lower regularity showed existence of local solutions. These solutions are global due to conservation of the $L^2(\T)$-norm and the dependence of the guaranteed time of existence on $\|w_0\|_{L^2(\T)}$.

\begin{theorem}[\cite{bourgain}]
    If $2 \leq p < 5$, the Cauchy problem \eqref{eq:NLSw} is locally wellposed in $L^4([-T,T]\times \T)$ for any $w_0 \in L^2(\T)$. 
    
    The guaranteed time of existence $T$ depends only on $\|w_0\|_{L^2(\R)}$.
\end{theorem}

Similar wellposedness results in $L^2$ for $2 \leq p < 5$ can be obtained with the help of $X^{s,b}$ spaces, see for example \cite[Section 3]{erdogan}. Using these techniques, the second author with Chaichenets, Hundertmark and Pattakos investigated the local wellposedness theory for $1 \leq p \leq 2$ \cite{qNLS}. Wellposedness in the mass-critical case $p = 5$ is an open problem \cite[p. 93]{erdogan}. In general, there is wellposedness in $H^s(\T), s\geq 0$ if $2 \leq p < 1+ \frac{4}{1-2s}$ \cite{bourgain}.

We only need well-posedness results in spaces $H^s(\T)$, $s \geq 1$, which are far away from being critical for any $p$. To prove global existence in $H^s(\T), s > 1$, we first prove a local result in $H^1(\T)$ which due to conservation of energy becomes immediately global, and then argue by persistence of regularity.

\begin{theorem}[Local wellposedness for $w$]\label{thm:lwpw}
    Given $p \geq 2$, the Cauchy problem \eqref{eq:NLSw} is locally wellposed in $C^0([-T,T],H^1(\T))$ for any $w_0 \in H^1(\T)$. 
    
    The guaranteed time of existence $T^*$ satisfies $T^* \gtrsim \|w_0\|_{H^1(\T)}^{1-p}$.
\end{theorem}
\begin{proof}
    Let $S(t) = e^{it\partial_x^2}$. The integral formulation of \eqref{eq:NLSw} is
    \begin{equation}\label{eq:integralformulationw}
        w(t) = S(t)w_0 + i\int_0^t S(t-s)(|w|^{p-1}w)(s) ds.
    \end{equation}
    We will show that the right-hand side is a contraction on
    \[X_{R,T} = \{v \in C([0,T],H^1(\T)): \|v\|_{C([0,T],H^1(\T))} \leq R\}\]
    where $R,T$ will be chosen later. Note that
    \[
        \partial_x \big(|f|^{p-1}f\big) = (p-1) |f|^{p-3}\real(\bar f f_x)f + |f|^{p-1}f_x
    \]
    Thus when $p \geq 1$, we find by Sobolev's embedding $H^1(\T) \subset L^\infty(\T)$ that given $f \in H^1(\T)$, $|f|^{p-1}f \in H^1(\T)$ with $\||f|^{p-1}f\|_{H^1(\T)} \lesssim \|f\|^p_{H^1(\T)}$. Together with the fact that $S(t)$ is an isometry on $H^s(\T)$ we can bound the $C([0,T],H^1(\R))$-norm of the right-hand side of \eqref{eq:integralformulationw} by
    \[  \|w_0\|_{H^1(\T)} + cTR^p \]
    for some constant $c >0$ if $w \in X_{R,T}$. Choosing $R = 2\|w_0\|_{H^1(\T)}$ deals with the first summand. For the second summand, we choose $T = (2c)^{-1}R^{1-p}$. This guarantees that the right-hand side defines a mapping on $X_{R,T}$. The contractive property is proven similarly: We have to estimate
    \begin{equation}\label{eq:diffintequation}
        \int_0^t S(t-s)\big(|w_1|^{p-1}w_1 - |w_2|^{p-1}w_2\big) \, ds
    \end{equation}
    uniformly in $H^1(\T)$. To this end, we bound
    \begin{equation*}
        ||w_1|^{p-1}w_1 - |w_2|^{p-1}w_2| \leq c|w_1 - w_2|\big(|w_1|^{p-1}+|w_2|^{p-2}\big)
    \end{equation*}
    and by adding a zero (see also Lemma \ref{lwpdifferenceestimate} for a more general estimate)
    \begin{equation*}
    \begin{split}
        |\partial_x(|w_1|^{p-1}w_1 - |w_2|^{p-1}w_2)| &\leq c\big(|w_1|^{p-1}|w_{1,x} - w_{2,x}| \\
        &\quad+ (|w_1|^{p-2} + |w_2|^{p-2})|w_1 - w_2||w_{2,x}|\big).
    \end{split}
    \end{equation*}
    This shows that we can estimate \eqref{eq:diffintequation} in $C([0,T],H^1(\R))$ by
    \[
        cTR^{p-1}\|w_1 - w_2\|_{H^1(\T)}.
    \]
    Hence the contractive property can be achieved by possibly making $T \sim R^{1-p}$ a bit smaller.
\end{proof}
Together with the conservation of the energy
\begin{equation}
    E(w) = \int_{\T}\frac{1}{2}|w_x|^2 + \frac{1}{p+1}|w|^{p+1} \, dx
\end{equation}
we obtain global wellposedness in $H^1(\T)$. Indeed,
\[
\begin{split}
    \|w_x(t)\|_{L^2}^2 - \|w_x(0)\|_{L^2}^2 &= \frac{2}{p+1}\big(\|w(0)\|_{L^{p+1}}^{p+1}-\|w(t)\|_{L^{p+1}}^{p+1}\big)\\
    &\lesssim \|w(0)\|_{L^2}^{\frac{p+3}{2}}\|w_x(0)\|_{L^2}^{\frac{p-1}{2}} \leq \|w_0\|_{H^1}^{p+1}.
\end{split}
\]
Here, we are able to argue without any smallness condition on the $L^2$ norm, because we are considering the defocusing equation. This shows
\begin{equation}\label{eq:energybound}
    \|w(t)\|_{H^1} \lesssim \|w_0\|_{H^1} + \|w_0\|_{H^1}^\frac{p+1}{2}.
\end{equation}

We turn to higher regularity and will prove that the constructed solutions are in $C([-T,T],H^s(\T))$ if the initial data is in $H^s(\T)$. As a byproduct, we obtain an exponential bound for these norms. 

We need a small technical result which is a standard result when $s$ is an integer and which was already used in the previous proof for $s=1$. We take the proof of \cite[Lemma A.9]{tao} and follow the explicit control on the constant in the estimate.

We introduce the notation $[p] = \sup\{k \in \mathbb{Z}, k \leq p\}$.

\begin{lemma}\label{lemma:fractionalestimate}
    Let $0 \leq s \leq [p]$. If $f \in H^s(\T)\cap L^\infty(\T)$, then also $|f|^{p-1}f \in H^s(\T)$ with bound
    \begin{equation}
        \||f|^{p-1}f\|_{H^s(\T)} \lesssim \|f\|_{L^\infty}^{p-1}\|f\|_{H^s(\T)}.
    \end{equation}
    If $p$ is an odd integer, the same holds for all $0 \leq s < \infty$.
\end{lemma}
\begin{proof}
    The case of $p$ being an odd integer follows from the classical estimate
    \[
        \|uv\|_{H^s(\T)} \lesssim \|u\|_{L^\infty(\T)}\|v\|_{H^s(\T)} + \|u\|_{H^s(\T)}\|v\|_{L^\infty(\T)}.
    \]
    Moreover, we may assume $0< s < [p]$ because in the case $s = [p]$ the norm can be easily estimated by calculating the weak derivatives directly.
    We denote by $P_N$ the usual Littlewood-Paley projection onto frequencies around the dyadic number $N$, which in the torus case is a projection onto a finite number of frequencies. Similarly, we define $P_{\geq N}$ and $P_{<N}$.
    
    First of all we note that
    \[
        \|P_{\leq 1} |f|^{p-1}f\|_{L^2} \leq \|f\|_{L^\infty}^{p-1}\|f\|_{L^2},
    \]
    meaning that we can restrict to frequencies $N > 1$ in the sum
    \[
        \||f|^{p-1}f\|_{H^s(\T)}^2 \sim_s \|P_{\leq 1}|f|^{p-1}f\|_{L^2(\T)}^2 + \sum_{N>1} N^{2s} \|P_N|f|^{p-1}f\|_{L^2}^2.
    \]
    We write $f = P_{<N}f + P_{\geq N}f$ and observe that
    \[
        \||f|^{p-1}f - |P_{<N}f|^{p-1}P_{<N}f\|_{L^2} \leq p\|P_{\geq N} f\|_{L^2}\|f\|_{L^\infty}^{p-1}
    \]
    since $(|x|^{p-1}x)' = p|x|^{p-1}$ and by the fundamental theorem of calculus. The high-frequency contribution can now be estimated since
    \begin{align*}
        \sum_{N>1} N^{2s} \|P_{\geq N} f\|_{L^2} &\leq \sum_{N' \geq N > 1} N^s (N')^s \|P_{N'}f\|_{L^2}^2 \lesssim \sum_{N' > 1} (N')^{2s}\|P_{N'}f\|_{L^2}^2.
    \end{align*}
    We turn to the low-frequency contribution and write for $k = [p]$,
    \[
        \|P_N(|P_{<N}f|^{p-1}P_{<N}f)\|^2_{L^2} \lesssim N^{-2k} \|\partial^k(|P_{<N}f|^{p-1}P_{<N}f)\|_{L^2}^2.
    \]
    From $|\partial^l(|x|^{p-1}x)| \lesssim_l |x|^{p-l}$ for $l \leq k$ and repeatedly using the chain rule, we see that
    \[
        |\partial^k(|P_{<N}f|^{p-1}P_{<N}f)| \lesssim_p \|f\|_{L^\infty}^{p-k}\sum_{r_1 + \dots + r_k = k} |\partial^{r_1}P_{<N}f|\dots|\partial^{r_k}P_{<N}f|.
    \]
    We can square this bound and estimate the squared sum on the right-hand side by the sum of the squares, losing a $p$-dependent factor. Because the number of such tuples depends only on $p$, we show the estimate for fixed $r_1, \dots, r_k$. By performing a Littlewood-Paley decomposition in each factor and giving up another combinatorical factor by ordering the frequencies, we see
    \begin{align*}
        &\|\partial^k(|P_{<N}f|^{p-1}P_{<N}f)\|_{L^2}^2 \\
        &\quad \lesssim_p \|f\|_{L^\infty}^{2p-2k}\sum_{N_1 \leq \dots \leq N_k < N} \|\partial^{r_1}P_{N_1} f\|_{L^\infty}^2\dots \|\partial^{r_{k-1}}P_{N_{k-1}}f\|_{L^\infty}^2\|\partial^{r_k}P_{N_k}f\|_{L^2}^2\\
        &\quad \lesssim \|f\|_{L^\infty}^{2p-2k}\sum_{N_1 \leq \dots \leq N_k < N} N_1^{2r_1}\dots N_k^{2r_k}\|P_{N_1} f\|^2_{L^\infty}\dots \|P_{N_{k-1}}f\|^2_{L^\infty}\|P_{N_k}f\|^2_{L^2}\\
        &\quad \lesssim \|f\|_{L^\infty}^{2p-2} \sum_{N' < N} (N')^{2k} \|P_{N'} f\|^2_{L^2},
    \end{align*}
    where from the third to the last line we estimated $\|P_{N_i}f\|_{L^\infty} \lesssim \|f\|_{L^\infty}$ to then do the summation first in $N_1$, then in $N_2$ and lastly in $N_{k-1}$, and rename $N_k = N'$. We conclude that
    \begin{align*}
        &\sum_{N}N^{2s}\|P_N(|P_{<N}f|^{p-1}P_{<N}f)\|^2_{L^2} \\
        &\qquad \lesssim \|f\|_{L^\infty}^{2(p-1)}\sum_N \sum_{N' < N} N^{2s-2k} (N')^{2k} \|P_{N'} f\|_{L^2}^2\\
        &\qquad \lesssim \sum_{N'} (N')^{2s} \|P_{N'}f\|_{L^2}^2
    \end{align*}
    by summing over $N$ first.
\end{proof}

\begin{theorem}[Global wellposedness for $w \in H^{s}(\T)$]\label{thm:gwpT}
    Let $p \geq 2$ and $w_0 \in H^s(\T)$ for $1 \leq s \leq [p]$. Then the global solution of Theorem \ref{thm:lwpw} is an $H^s(\T)$-solution and satisfies
    \begin{equation}\label{eq:exponentialbound}
         \|w(t)\|_{H^s} \leq e^{c(\|w_0\|_{H^1(\T)})t}\|w_0\|_{H^s(\T)},
    \end{equation}
    for some constant $c(\|w_0\|_{H^1(\T)})$. If $p$ is an odd integer, the same holds for $1 \leq s < \infty$.
\end{theorem}
\begin{proof}
    Let $w$ satisfy the integral equation \eqref{eq:integralformulationw} on some time interval $[0,T]$. We take the $H^s(\T)$ norm on both sides and estimate with the help of Lemma \ref{lemma:fractionalestimate}
    \begin{align*}
        \|w(t)\|_{H^s} &\leq \|w_0\|_{H^s(\T)} + \int_0^t \||w|^{p-1}w\|_{H^s(\T)} \, dt'\\
        &\leq \|w_0\|_{H^s(\T)} + c\|w\|_{L^\infty([0,T] \times \T)}^{p-1}\int_0^t \|w\|_{H^s(\T)}\, dt'.
    \end{align*}
    \eqref{eq:exponentialbound} now follows from Gronwall's lemma, where we get by Sobolev's inequality and the bound from energy conservation \eqref{eq:energybound},
    \begin{equation*}
         \|w(t)\|_{L^{^\infty}}^{p-1} \lesssim \|w(t)\|_{H^1}^{p-1} \lesssim \|w_0\|_{H^1(\T))}^{p-1} + \|w_0\|_{H^1(\T))}^{\frac{p^2-1}{2}} \sim c(\|w_0\|_{H^1(\T)})
    \end{equation*}
    for the constant in the exponential.
\end{proof}

We end this section by noting that the bound on the $H^s(\T)$ norm in Theorem \ref{thm:gwpT} is not optimal, but sufficient in our case to prevent blow-up. For example in the integrable case $p = 3$, there are infinitely many Hamiltonians which give immediate control over the $H^N$ norms for integer $N$.

\section{Local solutions on the line}\label{sec2}
To show wellposedness of NLS-type equations on the line in $L^2(\R)$, one usually makes use of Strichartz estimates. Indeed, the following result (both for the focusing and defocusing NLS) holds \cite[Theorem 2]{qNLS}:

\begin{theorem}[\cite{qNLS}]
    If $1 \leq p < 5$, the Cauchy problem \eqref{eq:NLSv} is locally wellposed in $C([0,T],L^2(\R))\cap L^{\frac{4(p+1)}{p-1}}([0,T],L^{p+1}(\R))$ for any $v_0 \in L^2(\R), w_0 \in H^1(\T)$. 
    
    In the case $1 \leq p < 5$, the guaranteed time of existence $T$ depends only on $\|v_0\|_{L^2(\R)}$ and $\|w_0\|_{H^1(\T)}$, whereas for $p = 5$ it depends on the profile of $v_0$ and $\|w_0\|_{H^1(\T)}$.
\end{theorem}

The restriction $p \leq 5$ comes from the fact that the problem \eqref{eq:NLS} is mass-supercritical when $p > 5$. On the other hand, we want to consider solutions in the energy space $H^1(\R)$, and so we do not run into a supercritical range when making $p$ large. Moreover, the wellposedness issue in $H^1(\R)$ is particularly easy because this space is a Banach algebra.

Our goal is to show wellposedness for the perturbed NLS on the line. Both in \cite{teeth} and \cite{qNLS}, local wellposedness was shown under the assumption of $w_0$ being more regular than $v_0$. This was due to the fact that the bound 
\[
    \|vw\|_{H^s(\R)} \lesssim \|v\|_{H^s(\R)}\|w\|_{H^{s+1}(\T)}
\]
was used (see for example \cite[Lemma 11]{teeth}). A close inspection of the argument shows that a bound against $\|w\|_{H^{s+1/2+}(\T)}$ would also work. On the other hand, the case $s=1$ suggests that by localizing in space and using periodicity, one can also estimate with the same regularity $s$, because for example
\[
\begin{split}
    \|vw'\|^2_{L^2(\R)} &= \sum_{k \in \Z} \int_k^{k+1} |v|^2 |w'|^2 \, dx \leq \|w'\|_{L^2(\T)}^2 \sum_{k \in \Z} \|v\|_{L^\infty([k,k+1))}^2 \\
    &\lesssim \|w\|_{H^1(\T)}\sum_k \|v\|^2_{H^1([k,k+1))} = \|w\|_{H^1(\T)}\|v\|_{H^1(\T)}.
\end{split}
\]
This gives a strictly better bound than putting $w'$ into $L^\infty(\T) = L^\infty(\R)$ in the first place. With this remark at hand, and using the arguments from \cite{teeth} and \cite{qNLS}, we will obtain the local wellposedness results with initial data in $H^1(\R) + H^1(\T)$.

We need estimates on difference terms. This is done in the next lemma.

\begin{lemma}\label{lwpdifferenceestimate}
    Let $p \geq 2$ and 
    \[G(v_1,v_2,w) = |v_1 +w|^{p-1}(v_1+w) - |v_2 + w|^{p-1}(v_2 + w).\]
    There exists a constant $c>0$ such that the following estimates hold true:
    \begin{equation}
        |G(v_1,v_2,w)|\leq c|v_1 - v_2||(v_1,v_2,w)|^{p-1},
    \end{equation}
    \begin{equation}
        \begin{split}
            |G(v_1,v_2,w)_x| \leq &c|v_{1,x} - v_{2,x}||(v_1,v_2,w)|^{p-1} + \\
            &\qquad c|v_1 - v_2||(v_{1,x},v_{2,x},w_x)||(v_1,v_2,w)|^{p-2}.
        \end{split}
    \end{equation}
    The same holds true if we replace the function $|x|^{p-1}x$ by $|x|^p$ in all of the arguments.
\end{lemma}
\begin{proof}
    Note that $\partial_s |f(s)|^p = p |f(s)|^{p-2}\real(\bar f \partial_s f)$. By the fundamental theorem of calculus we can write with $v(s) = v_2+s(v_1-v_2) + w$
    \begin{equation*}
    \begin{split}
        G(v_1,v_2,w)&= (v_1 - v_2) \int_0^1 |v(s)|^{p-1}\, ds\\
        &\quad + (p-1)\int_0^1 |v(s)|^{p-3}v(s)\real(v(s))\overline{(v_1-v_2)})\, ds
    \end{split}
    \end{equation*}
    Since $|v_2+s(v_1-v_2) + w|^{p-1} \lesssim |v_1|^{p-1} + |v_2|^{p-1} + |w|^{p-1}$, this shows the first estimate. For the second estimate, the first summand takes care of when the derivative falls on $(v_1 - v_2)$. Moreover,
    \begin{equation*}
        \begin{split}
            \Big|\partial |v(s)|^{p-1}\Big| \lesssim (|v_{1,x} + |v_{2,x}| + |w_x|)(|v_1|^{p-2} + |v_2|^{p-2} + |w|^{p-2})
        \end{split}
    \end{equation*}
    which produces the second summand in the estimate coming from the derivative falling on the integrand. The case of $|x|^{p}$ instead of $|x|^{p-1}x$ is proven analogously. 
\end{proof}

\begin{theorem}\label{thm:nlsvlwp}
    Let $p \geq 2$. The Cauchy problem \eqref{eq:NLSv} is locally wellposed in $C([0,T],H^1(\R))$ for any $v_0 \in H^1(\R), w_0 \in H^s(\T), s \geq 1$. 
    
    The guaranteed time of existence $T^*$ depends only on the $H^1$ norms $\|v_0\|_{H^1(\R)}$ and $\|w_0\|_{H^{1}(\T)}$. More precisely,
    \[
        T^* \gtrsim \min\big(\|v_0\|_{H^1(\R)}^{1-p},\|w_0\|_{H^1(\T)}^{1-p},\|w_0\|_{H^1(\T)}^{-\frac{p^2-1}2}\big).
    \]
\end{theorem}

\begin{proof}
    The proof is a standard Banach fixed point argument. If we let $S(t) = e^{it\partial_x^2}$, then the integral formulation of \eqref{eq:NLSv} is
    \begin{equation}\label{eq:integralformulation}
        v(t) = S(t)v_0 + i\int_0^t S(t-t')(|v+w|^{p-1}(v+w)-|w|^{p-1}w)\, dt'.
    \end{equation}
    We will show that the right-hand side is a contraction on
    \[X_{R,T} = \{v \in C([0,T],H^1(\R)): \|v\|_{C([0,T],H^1(\R))} \leq R\}\]
    where $R,T$ will be chosen later. We claim that
    \begin{equation}\label{eq:LWPestimate}
            \Vert |v+w|^{p-1}(v+w) - |w|^{p-1}w \Vert_{H^1(\R)} \lesssim \|v\|_{H^1(\R)}(\|w\|_{H^{1}(\T)}^{p-1} + \|v\|_{H^1(\R)}^{p-1})
    \end{equation}
    Indeed by Lemma \ref{lwpdifferenceestimate} with $v_2 = 0$,
    \begin{equation*}
            \Big| |v+w|^{p-1}(v+w) - |w|^{p-1}w\Big| \lesssim |v|\Big(|v|^{p-1} + |w|^{p-1}\Big),
    \end{equation*}
    and
    \begin{equation*}
        \begin{split}
            &\Big|\partial\big(|v+w|^{p-1}(v+w) - |w|^{p-1}w\big)\Big| \\
            &\qquad \lesssim |v_x|\Big(|v|^{p-1} + |w|^{p-1}\Big) + |v||w_x|\Big(|v|^{p-2} + |w|^{p-2}\Big).
        \end{split}
    \end{equation*}
    The $L^2(\R)$ norm of the first term and the first summand of the second term are seen to be bounded by the right-hand side of \eqref{eq:LWPestimate}  simply by putting $w \in L^\infty(\R)$ and Hölder. For the second summand, we localize in space and use Sobolev's inequality on the interval $[k,k+1]$ to find
    \[
    \begin{split}
        &\int_\R |v|^2 |w_x|^2 \big(|v|^{p-2} + |w|^{p-2}\big)^2 \, dx \\
        &\qquad \leq \|w_x \|^2_{L^2(\T)}\big(\|v\|_{L^\infty(\R)}^{p-2} + \|w\|_{L^\infty(\R)}^{p-2}\big)^2 \sum_k \|v\|^2_{L^\infty([k,k+1))}\\
        &\qquad \lesssim \|w\|^2_{H^1(\T)}\|v\|_{H^1(\R)}^2\big(\|v\|_{H^1(\R)}^{p-2} + \|w\|_{H^1(\T)}^{p-2}\big)^2.
    \end{split}
    \]
    This proves \eqref{eq:LWPestimate}
    
    For the $H^{1}(\T)$ norm of $w$ we have the energy bound \eqref{eq:energybound}, and together with the fact that $S(t)$ is an isometry on $H^1(\R)$ we can bound the $C([0,T],H^1(\R))$-norm of the right-hand side of \eqref{eq:integralformulation} by
    \[  \|v_0\|_{H^1(\R)} + cTR\big(\|w_0\|^{p+1}_{H^1(\T)}+\|w_0\|_{H^1(\T)}^{\frac{p^2-1}2} + R^{p-1}\big)
    \]
    if $v \in X_{R,T}$. Choosing $R = 2\|v_0\|_{H^1(\R)}$ deals with the first summand, and for the second summand, we let 
    \[
    T \lesssim \min\big(\|w_0\|_{H^1(\T)}^{1-p},\|w_0\|_{H^1(\T)}^{-\frac{p^2-1}2}, R^{1-p}\big)
    \]
    be small enough. This guarantees that the right-hand side of \eqref{eq:integralformulation} defines a mapping on $X_{R,T}$. The contractive property is proven similarly: If we keep the notation from Lemma \ref{lwpdifferenceestimate} and use it, then
    \[
        |G(v_1,v_2,w)| \lesssim |v_1-v_2|(|v_1|^{p-1} + |v_2|^{p-1} + |w|^{p-1}),
    \]
    and
    \[
    \begin{split}
        |G(v_1,v_2,w)_x| &\lesssim |v_{1,x}-v_{2,x}|(|v_1|^{p-1}+|v_2|^{p-1}+|w|^{p-1}) \\
        &\quad + |v_{1}-v_{2}|(|v_{1,x}|+|v_{2,x}|+|w_x|)(|v_1|^{p-2}+|v_2|^{p-2}+|w|^{p-2})
    \end{split}
    \]
    When there is no derivative term falling on $w$, the $L^2(\R)$ norm of $G$ respectively $G_x$ can be estimated by putting $w$ in $L^\infty$. The worst term is
    \[
    \begin{split}
        &\int |v_1-v_2|^2 |w_x|^2(|v_1|^{p-2}+|v_2|^{p-2}+|w|^{p-2})^2\,dx \\
        &\qquad \lesssim (\|v_1\|_{L^\infty}^{p-2}+\|v_2\|_{L^\infty}^{p-2}+\|w\|_{L^\infty}^{p-2})^2\sum_k \|v_1-v_2\|^2_{L^\infty([k,k+1))}\|w_x\|^2_{L^2(\T)},
    \end{split}
    \]
    and we can estimate as before by Sobolev's embedding. Hence we find
    \begin{equation*}
    \begin{split} 
        &\Big\|\int_0^t S(t-t') G(v_1,v_2,w)\, dt'\Big\|_{H^1(\R)} \\
        &\qquad \lesssim \|v_1-v_2\|_{L^\infty([0,T], H^1(\R))}T(R^{p-1}+\|w\|^{p-1}_{L^\infty([0,T], H^1(\T))}) \\
        &\qquad \lesssim \|v_1-v_2\|_{L^\infty([0,T], H^1(\R))}T(R^{p-1}+\|w_0\|^{p-1}_{H^1(\T)}+\|w_0\|^{\frac{p^2-1}{2}}_{H^1(\T)})
    \end{split} 
    \end{equation*}
    In particular with the same relative smallness condition as before, we obtain a contraction on $X_{R,T}$.
\end{proof}

As a byproduct of Theorem \ref{thm:nlsvlwp} we obtain a blow-up alternative for the solution $v$ to \eqref{eq:NLSv}: Denote by $T^*$ the maximal time of existence. Then either $T^* < \infty$ and
\begin{equation}\label{eq:blowupalternative}
    \limsup_{t \to T^*} \|v(t)\|_{H^1(\R)} = \infty,
\end{equation}
or $T^* = \infty$.

Indeed, we see that if we had a maximal solution of \eqref{eq:NLSv} with the property $\limsup_{t \to T^*} \|v(t)\|_{H^1(\R)} < \infty$, we could continue it to some time $T^* + \delta$ by Theorem \ref{thm:nlsvlwp}, yielding a contradiction to its definition as a maximal solution.

\section{Global solutions on the line}\label{sec3}

In this section, we will prove our main theorem. To this end, we define momentum $M$, energy $E$ and Hamiltonian $H$ as
\begin{equation*}
    M(v) = \int \frac{1}{2} |v|^2 \, dx, \qquad E(v) = \int \frac{1}{2} |v_x|^2 + \frac{1}{p+1}|v|^{p+1}.
\end{equation*}
and
\begin{equation*}
    H(v) = \int \frac{1}{2} |v_x|^2 + \frac{1}{p+1}\big(|v+w|^{p+1} - |w|^{p+1} - (p+1)|w|^{p-1}\real(v\bar w) \big) \, dx.
\end{equation*}
Moreover, we introduce the notation $(f,g) = \real \int_\R f(x) \bar g(x) \, dx$.

\begin{theorem}\label{thm:mainformal}
    Let $v \in C^0([0,T),H^1(\R))$ be a solution of \eqref{eq:NLSv}. Let $p \geq 3$ and $w_0 \in H^{s}(\T)$ with 
    \begin{itemize}
        \item $3/2 < s < \infty$, if $p = 3$,
        \item $5/2 < s < \infty$, if $p \geq 5$ is an odd integer and 
        \item $5/2 < s \leq [p]$, else.
    \end{itemize}
    Then there is a constant $C = C(T, \|v_0\|_{H^{1}},\|w_0\|_{H^{s}}) > 0$ such that
	\begin{equation}
	    \sup_{t \in [0,T)} M(v(t)) + E(v(t)) \leq C
	\end{equation}
	In particular, \eqref{eq:NLSv} is globally wellposed.
\end{theorem}

The idea of the proof is that by \eqref{eq:flow}, we know that there are no derivatives falling on $v$, and we essentially just have to count factors of $v$ in the time derivative of $H$. Factors that depend on the $L^\infty$ norm of $w$ and its derivatives are allowed since by Theorem \ref{thm:gwpT} it is bounded locally in time. 

We will see that there are always more than two, but never too many factors of $v$. If there are too many factors of $v$, we are not able to estimate by $M^\alpha E^{1-\alpha}$ anymore, leading to a break down in the Gronwall argument. This is also the reason of the higher regularity assumption $w_0 \in H^{5/2+}(\T)$ in the case $p > 3$, because the estimate
\[
    \int |v_x||v|^{p-1}\, dx \lesssim \|v_x\|_{L^2}^{2\alpha}\|v\|_{L^{p+1}}^{(1-\alpha)(p+1)}
\]
only works for $p \leq 3$.

Before proving Theorem \ref{thm:mainformal}, we need additional estimates in the spirit of Lemma \ref{lwpdifferenceestimate}.

\begin{lemma}\label{lemma:estimates}
Let $p \geq 3$. There exists a constant $c = c(p) > 0$ such that the following estimates hold true:
    \begin{equation}
        \big| |v+w|^{p-1} - |w|^{p-1}- (p-1)\real(w\bar v)|w|^{p-3}\big| \leq c|v|^2(|w|^{p-3} + |v|^{p-3}),
    \end{equation}
    \begin{equation}
    \begin{split}
        &\big| |v+w|^{p+1} - |w|^{p+1} - (p+1)\real(w\bar v)|w|^{p-1} - |v|^{p+1} \big| \\
        &\qquad \leq c|v|^2|w|(|w|^{p-2} + |v|^{p-2}),
    \end{split}
    \end{equation}
    \begin{equation}
        \begin{split}
            &\big| |v+w|^{p-1}(v+w) - |w|^{p-1}w  - ((p-1) \real(w\bar v)|w|^{p-3}w + v|w|^{p-1}) \big| \\
        &\qquad \leq  c|v|^2(|w|^{p-2} + |v|^{p-2}).
        \end{split}
    \end{equation}
    In particular, there exists a constant $c = c(p, \|w\|_{L^\infty(\T)}) >0$ such that
    \begin{equation}\label{eq:energyequivalence}
        H \leq c M + E, \qquad E \leq cM + H.
    \end{equation}
\end{lemma}
\begin{proof}
    We define $f(s,t) = |sv+tw|^q$. Note first that
    \begin{align*}
        \partial_s f(s,t) &= q|sv+tw|^{q-2}\real((sv+tw)\bar v),\\
        \partial_t f(s,t) &= q|sv+tw|^{q-2}\real((sv+tw)\bar w),\\
        \partial_s^2 f(s,t) &= q|sv+tw|^{q-2}|v|^2 + q(q-2)|sv+tw|^{q-4}\real((sv+tw)\bar v)^2,\\
        \partial_{s}\partial_t f(s,t) &= q|sv+tw|^{q-2}\real(v\bar w) \\
        &\qquad + q(q-2)|sv+tw|^{q-4}\real((sv+tw)\bar v)\real((sv+tw)\bar w),\\
        \partial_s^2 \partial_t f(s,t) &= q(q-2)|sv+tw|^{q-4}\real((sv+tw)\bar v)\real(v\bar w) \\
        &\qquad + q(q-2)|sv+tw|^{q-4}|v|^2\real((sv+tw)\bar w),\\
        &\qquad + q(q-2)|sv+tw|^{q-4}\real((sv+tw)\bar v)\real(v\bar w),\\
        &\qquad + q(q-2)(q-4)|sv+tw|^{q-6}\real((sv+tw)\bar v)^2\real((sv+tw)\bar w).
    \end{align*}
    In particular, we see that for $0 \leq s,t \leq 1$,
    \begin{align*}
        |\partial_s^2 f(s,t)| &\lesssim |v|^2(|v|^{q-2}+|w|^{q-2}),\\
        |\partial_s^2 \partial_t f(s,t)| &\lesssim |v|^2|w|(|v|^{q-3} + |w|^{q-3}).
    \end{align*}
    Now the left-hand side of the first estimate is $f(1,1) - f(0,1) - \partial_s f(0,1)$ with $q = p-1$, and so the first estimate follows from the fundamental theorem of calculus,
    \begin{equation*}
        \begin{split}
            |f(1,1) - f(0,1) - \partial_s &f(0,1)| = \Big|\int_0^1 \partial_s f(s,1) - \partial_s f(0,1) \, ds\Big| \\
            &= \Big|\int_0^1 \int_0^s \partial_s^2 f(s',1)\, ds' ds\Big| \lesssim |v|^2(|v|^{q-2}+|w|^{q-2}).
        \end{split}
    \end{equation*}
    For the second estimate we note that $f(0,0) = \partial_s f(0,0) = 0$ and use the fundamental theorem of calculus three times to see
    \begin{equation*}
        \begin{split}
            |f(1,1) - f(0,1) - f(1,0) - \partial_s f(0,1)| &= \Big| \int_0^1\int_0^1 \int_0^s  \partial_s^2 \partial_t f(s',t)\, ds' ds dt \Big|\\
            &\lesssim |v|^2|w|(|v|^{q-3} + |w|^{q-3}).
        \end{split}
    \end{equation*}
    The third estimate follows similarly by arguing with $g(s,t) = |sv + tw|^{p-1}(sv + tw)$.
    With these estimates, we use Hölder $\|v\|_{L^p}^p \leq \|v\|_{L^2}^{2/(p-1)}\|v\|_{L^{p+1}}^{(p+1)(p-2)/(p-1)}$ and Young to see
    \begin{equation*}
    \begin{split}
         |H(v) - E(v)| &\lesssim \int |v|^2|w|(|w|^{p-2} + |v|^{p-2}) \, dx \\
            &\lesssim \|w\|_{L^\infty(\T)}^{p-1}M(v) + \|w\|_{L^\infty(\T)}\|v\|_{L^p(\R)}^p\\
            &\lesssim \|w\|_{L^\infty(\T)}^{p-1}M(v) + \|w\|_{L^\infty(\T)}M(v)^{\frac{1}{p-1}}E(v)^{\frac{p-2}{p-1}} \\
            &\lesssim \varepsilon E(v) + C(\varepsilon) M(v).
    \end{split}
    \end{equation*}
    Choosing $\varepsilon$ small enough, we arrive at \eqref{eq:energyequivalence}.
\end{proof}

In order to prove Theorem \ref{thm:mainformal}, we want to take time derivatives of $E$ and $M$ and hence of $v$. If $v_0 \in H^1(\R)$, then $v(t) \in H^1(\R)$ and hence from \eqref{eq:NLSv} we see that $v_t \in H^{-1}(\R)$ for all times. This is enough to rigorously calculate $\partial_t M = (v,v_t)$, by interpreting the involved integral as a dual pairing between $H^1$ and $H^{-1}$. For the bilinear part of the energy, $-(v_{xx}, v_t)$, this does not suffice any more. Our solution is to employ a twisting trick (see for example \cite{bosegases} and \cite{dispersionmanaged}) and to work in the interaction picture, that is with the function $\psi(t) = e^{-it\partial_x^2} v(t)$.

\begin{proof}[Proof of Theorem \ref{thm:mainformal}]
    Fix $T > 0$. We use Gronwall's lemma to obtain an exponential bound on $M + H$. By \eqref{eq:energyequivalence}, this is enough to bound $M + E$ and hence the $H^1$ norm. 
    
    First of all, note that for any $2 \leq q \leq p+1$, we have by Hölder and Young
    \[
        \|v\|_{L^q}^q \leq \|v\|_{L^2}^{2\frac{p-q+1}{p-1}}\|v\|_{L^{p+1}}^{(p+1)\frac{q-2}{p-1}} \leq M(v)^{\frac{p-q+1}{p-1}}E(v)^{\frac{q-2}{p-1}} \leq M(v) + E(v).
    \]
    This shows that powers of $v$ ranging from two to $p+1$ are allowed.
    
    We begin with $M$ and see, interpreting the integrals in the first line as a dual pairing between $H^1$ and $H^{-1}$,
    \begin{align*}
        \partial_t M(v) &= (v, v_t) = (iv, -v_{xx} + |v+w|^{p-1} (v+w) - |w|^{p-1}w) \\
        &= (iv,|v+w|^{p-1} (v+w) - |w|^{p-1}w) \\
        &\lesssim E(v) + M(v) \lesssim H(v) + M(v).
    \end{align*}
    Here, the summand $(iv, -v_{xx})$ vanishes by partially integrating once and we used Lemma \ref{lwpdifferenceestimate} in the last line. To calculate the time derivative of $H$, we start with a formal calculation for the bilinear part,
    \begin{equation*}
        \partial_t \frac{1}{2} \int |v_x|^2 \, dx = (v_t, -v_{xx}) = -\big(v_t, |v+w|^{p-1}(v+w) - |w|^{p-1}w\big),
    \end{equation*}
    using \eqref{eq:NLSv} to rewrite $-v_{xx}$ and $(v_t, iv_t) = 0$. As $v_t \in H^{-1}$ and $v_{xx} \in H^{-1}$, their product is not well defined and the middle step in the above calculation needs to be justified. To make it rigorous, we define $\psi(t) = e^{-it\partial_x^2} v(t) = S(-t)v(t)$. Recall that $v(t)$ satisfies \eqref{eq:integralformulation}, and hence $\psi(t)$ satisfies
    \[
        \psi(t) = v_0 + \int_0^t S(-t')\big(|v+w|^{p-1}(v+w) - |w|^{p-1}w\big) \, dt'.
    \]
    In particular, we see that 
    \[i\partial_t \psi = S(-t)(|v+w|^{p-1}(v+w) - |w|^{p-1}w)(t),\]
    which shows $\psi \in C^1((0,T), H^1(\R))$. Since $S(-t)$ is an isometry on $L^2$ and commutes with derivatives, we calculate
    \begin{align*}
        \partial_t \frac{1}{2}&\int |v_x|^2 = -(\psi_{xx}, \psi_t) \\
        &= -\big(i\psi_{xx}, S(-t) (|v+w|^{p-1}(v+w) - |w|^{p-1}w)\big)\\
        &= -\big(iv_{xx}, |v+w|^{p-1}(v+w) - |w|^{p-1}w\big)\\
        &= -\big(v_t + i(|v+w|^{p-1}(v+w) - |w|^{p-1}w), |v+w|^{p-1}(v+w) - |w|^{p-1}w\big)\\
        &= -\big(v_t, |v+w|^{p-1}(v+w) - |w|^{p-1}w\big).
    \end{align*}
    which is well-defined as a dual pairing, because $|v+w|^{p-1}(v+w) - |w|^{p-1}w \in H^1(\R)$ by Lemma \ref{lwpdifferenceestimate}.
    
    For the nonlinear term of the Hamiltonian, we calculate
    \begin{align*}
        \partial_t \frac{1}{p+1}& \int |v+w|^{(p+1)} - |w|^{p+1} - (p+1)\real(w\bar v)|w|^{p-1} \, dx \\
        &= \real\int\big( |v+w|^{p-1}(v_t+w_t)(\bar v + \bar w) - |w|^{p-1}(w_t \bar w) \\
        &\qquad \qquad  -(w_t \bar v |w|^{p-1} + v_t\bar w |w|^{p-1} + (p-1)w \bar v \real (w_t \bar w) |w|^{p-3}\big)\\
        &= \real\int \big(|v+w|^{p-1}(\bar v + \bar w)v_t - |w|^{p-1} \bar w v_t\big) + \int R,
    \end{align*}
    where the remainder $R$ only carries time derivatives on $w$,
    \begin{equation*}
        R = \real\big(|v+w|^{p-1}(\bar v + \bar w) w_t - |w|^{p-1}(\bar w+\bar v) w_t\big) - (p-1)|w|^{p-3} \real(w_t \bar w)\real (w\bar v).
    \end{equation*}
    Since the first summand cancels with the time derivative of the bilinear part, we arrive at $\partial_t H = \int R$ as predicted by \eqref{eq:flow}. We conclude
    \begin{align*}
        \partial_t H &=  \big(w_t, |v+w|^{p-1}(v+w) - |w|^{p-1}(v+w) - (p-1)\real(w \bar v)|w|^{p-3}w\big)\\
        &= \big(w_t, v(|v+w|^{p-1}-|w|^{p-1})\big) \\
        &\qquad \qquad + \big(w_t, w(|v+w|^{p-1} - |w|^{p-1} - (p-1)\real(w\bar v)|w|^{p-3})\big).
    \end{align*}
    We first argue how to handle this term in the case $p > 3$. In this case, we assumed $w_0 \in H^{s}(\T)$, $s > 5/2+$ with corresponding upper bound depending on whether $p$ is an odd integer or not. This means by Theorem \ref{thm:gwpT} that we have local in time boundedness in $L^\infty$ of $w_{xx}$, hence of $w_t$. Now using Lemma \ref{lemma:estimates}, this implies \begin{align*}
        \big|\partial_t H\big| &\lesssim \|w_t\|_{L^\infty}\int |v|^2(|v|^{p-2}+|w|^{p-2})\, dx\\
        &\qquad + \|w_t\|_{L^\infty}\|w\|_{L^\infty}\int |v|^2(|v|^{p-3}+|w|^{p-3})\, dx\\
        &\lesssim E(v) + M(v) \lesssim H(v) + M(v),
    \end{align*}
    with a bound depending on $p, w_0, T$. Gronwall gives
    \begin{equation*}
        H(v(t)) + M(v(t)) \lesssim (H(v_0) + M(v_0))e^{Ct},
    \end{equation*}
    and proves the theorem for this case.
    
    We turn to $p = 3$. By plugging in the equation \eqref{eq:NLSw} for $w_t$, we obtain two terms for each summand in $\partial_t H$, one with $|w|^{2}w$ and one with $w_{xx}$. For the term with $|w|^{2}w$, we estimate $w$ in $L^\infty$ and argue as above. The other two terms are
    \begin{align*}
        &\big(iw_{xx}, v(|v+w|^2-|w|^2)\big) + \big(iw_{xx}, w(|v+w|^2 - |w|^2 - 2\real(w\bar v))\big)\\
        &\qquad = \big(iw_{xx}, |v|^2v + 2v\real(v\bar w) + |v|^2 w\big).
    \end{align*}
    We integrate by parts once. This produces terms where the derivative falls on a copy of $w$, and terms where the derivative falls on a copy of $v$. The former case is handled just as above because $\|w_x\|_{L^\infty}$ is bounded locally in time by $w_0 \in H^{3/2+}$. In the latter case, we put $v_x$ in $L^2$ and estimate
    \begin{align*}
        &\Big|\big(iw_{x}, (|v|^2v)_x + 2v_x\real(v\bar w) + 2v\real(v_x\bar w) + (|v|^2)_x w\big)\Big| \\
        &\qquad \qquad \lesssim \|w_x\|_{L^\infty}\big(\|v_x\|_{L^2}\|v\|_{L^4}^2 + \|w\|_{L^\infty}\|v_x\|_{L^2}\|v\|_{L^2}\big)\\
        &\qquad \qquad \lesssim E(v) + E(v)^{\frac{1}{2}}M(v)^{\frac{1}{2}}
    \end{align*}
    which can be estimated by $E(v) + M(v)$ by Young. As before we conclude via Gronwall and finish the proof.
\end{proof}

\end{document}